\documentclass{amsart}
\usepackage{proof}
\usepackage{dst}
\usepackage{ast2}

\usepackage[margin=1.4in]{geometry}

\newcommand{\IQC}{\ensuremath{ \usftext{IQC}  }}
\newcommand{\CQC}{\ensuremath{ \usftext{CQC}  }}
\newcommand{\Lo}{\ensuremath{ \usftext{L}  }}
\newcommand{\MQC}{\ensuremath{ \usftext{MQC}  }}

\usepackage{amssymb,amsmath,amsthm}
\usepackage[foot]{amsaddr}
\title{A Kuroda-style $j$-translation}
\author{Benno van den Berg}
  \address{Institute for Logic, Language and Computation (ILLC), University of Amsterdam, P.O. Box 94242, 1090 GE Amsterdam, the Netherlands. E-mail: bennovdberg@gmail.com.}
\date{\today}

\begin{document}

\begin{abstract}
 In topos theory it is well-known that any nucleus $j$ gives rise to a translation of intuitionistic logic into itself in a way which generalises the G\"odel-Gentzen negative translation. Here we show that there exists a similar $j$-translation which is more in the spirit of Kuroda's negative translation. The key is to apply the nucleus not only to the entire formula and universally quantified subformulas, but to conclusions of implications as well. The development is entirely syntactic and no knowledge of topos theory is required to read this small note. 
\end{abstract}

\maketitle

\section{Introduction}

Since everything which is provable in intuitionistic logic is also classically provable, but not \emph{vice versa}, it is natural to think of intuitionistic logic as a weakening of classical logic. Nevertheless, it is also possible to think of intuitionistic logic as an extension of classical logic, because there exist faithful copies of classical logic inside intuitionistic logic. Such an isomorphic copy of classical logic inside intuitionistic logic is called a \emph{negative translation}.

Let us define a negative translation more formally as a mapping $\varphi \mapsto \varphi^*$ which sends formulas in predicate logic to formulas in predicate logic such that the following two statements hold for all formulas $\varphi, \varphi_1,\ldots,\varphi_n$ and $\psi$:
\begin{enumerate}
\item $\vdash_{\CQC} \varphi \leftrightarrow \varphi^*$
\item If $\varphi_1,\ldots,\varphi_n \vdash_{\CQC} \psi$, then $(\varphi_1)^*,\ldots,(\varphi_n)^* \vdash_{\IQC} \psi^*$.
\end{enumerate}
(Here $\CQC$ stands for classical predicate logic, while $\IQC$ stands for intuitionistic predicate logic. Note that (1) implies that the converse of (2) holds as well.) One well-known negative translation is the G\"odel-Gentzen negative translation which puts double negations ($\lnot\lnot$) in front of atomic formulas, disjunctions and existential quantifiers. A lesser known negative translation is due to Kuroda which puts double negations in front of the entire formula and behind universal quantifiers \cite{kuroda51}. In a sense, these two translations are only syntactic variants of each other, because the G\"odel-Gentzen and Kuroda negative translation of some formula are intuitionistically equivalent. Since these negative translation extend to systems stronger than logic (like arithmetic, for instance), they can be used to prove equiconsistency and conservativity results for classical systems over intuitionistic analogues.

In topos theory the G\"odel-Gentzen negative translation has been generalised to general \emph{nuclei} (also known as local operators, or Lawvere-Tierney topologies). Essentially, a nucleus is an operation on the collection of truth values which is monotone, inflationary, idempotent and commutes with conjunction, with double negation being a prime example. Thinking of such a nucleus $j$ as an operation on formulas, we can define for each formula $\varphi$ a new formula $\varphi^j$ which is obtained by applying the $j$-operation to atomic formulas, disjunctions and existential quantifiers, as in the G\"odel-Gentzen negative translation. The result is no longer a negative translation; however, because we do still have
\[ \varphi_1,\ldots,\varphi_n \vdash_{\IQC} \psi \Longrightarrow (\varphi_1)^j,\ldots,(\varphi_n)^j \vdash_{\IQC} \psi^j, \]
it provides a translation of intuitionistic logic into itself (often called the $j$-translation).

This much is well known. What does not seem to have been observed before is that something similar is possible for the Kuroda negative translation. So let $j$ be a nucleus and let $\varphi_j$ be the result of putting $j$ in front of the entire formula, behind universal quantifiers and \emph{in front of conclusions of implications}. Then we have
\[ \varphi_1,\ldots,\varphi_n \vdash_{\IQC} \psi \Longrightarrow (\varphi_1)_j,\ldots,(\varphi_n)_j \vdash_{\IQC} \psi_j, \]
so we again have a translation of intuitionistic logic into itself, a Kuroda-style $j$-translation. In fact, $\varphi_j$ and $\varphi^j$ are intuitionistically equivalent. The fact that $j$ also has to be applied to conclusions of implications is related to the fact that the Kuroda negative translation does not work as a translation of classical logic into minimal logic (constructive logic without the \emph{ex falso} rule $\bot \to \varphi$), but the modified Kuroda translation which also puts double negations in front of conclusions of implications does (as was apparently first observed by Ferreira and Oliva, see 
\cite[Section 6]{ferreiraoliva12}). We will explain all of this more precisely below. 

It turns out that the Kuroda-style $j$-translation explains some initially puzzling phenomena in the literature, of which I will give a few examples, so I hope some readers may find this little note illuminating. I will not give many proofs, because most of them are routine inductions on the structure of a formula or a derivation.

\section{Nuclei}

Throughout this paper $\Lo$ is a logic which could be either $\CQC$ (classical predicate logic), $\IQC$ (intuitionistic predicate logic), or $\MQC$ (minimal predicate logic). We assume that these logics have been formulated with $\land, \lor, \to, \exists, \forall$ and $\bot$ as primitive, while $\lnot \varphi$ is defined as $\varphi \to \bot$. Minimal logic can be axiomatised in natural deduction-style using the standard introduction and elimination rules for the logical connectives, and with $\varphi$ following from $\varphi$ as the only axiom. In minimal logic $\bot$ has no special properties, so acts as an arbitrary propositional constant. Intuitionistic logic is obtained from minimal logic by adding \emph{ex falso} $\bot \to \varphi$, while classical logic is obtained from intuitionistic logic by adding the law of excluded middle $\varphi \lor \lnot\varphi$ or double negation elimination $\lnot\lnot\varphi \to \varphi$.

\begin{defi}{nucleus}
Let $j$ be a function mapping formulas in predicate logic to formulas in predicate logic. Such a mapping will be called a \emph{nucleus} (relative to the logic $\Lo$) if for all formulas $\varphi$ and $\psi$ the following statements are provable in $\Lo$:
\[ \begin{array}{l}
\vdash_{\Lo} \varphi \to j \varphi \\
\vdash_{\Lo} j (\varphi \land \psi) \leftrightarrow ( \, j \varphi \land j \psi \, ) \\
\vdash_{\Lo} (\varphi \to j \psi) \to (j \varphi \to j \psi) \\
\vdash_{\Lo} (j\varphi)[t/x] \leftrightarrow j(\varphi[t/x])
\end{array} \]
\end{defi}

\begin{exam}{examplesofnuclei} The main example of a nucleus is the double-negation nucleus \[ j \varphi := \lnot\lnot \varphi; \] more generally, we have that $j \varphi := (\varphi \to A) \to A$ is a nucleus for any fixed sentence $A$ (the double-negation nucleus being the special case where $A = \bot$). Other examples of nuclei are 
\begin{eqnarray*}
j \varphi & := & \varphi \lor A \\
j \varphi & := & A \to \varphi \\
j \varphi & := & (\varphi \to A) \to \varphi
\end{eqnarray*} for any fixed sentence $A$. Note that all of these examples are nuclei over $\MQC$ already.
\end{exam}

\begin{lemm}{basicpropofnuclei}
For any nucleus we can prove:
\begin{displaymath}
\begin{array}{l}
\vdash_{\Lo} (\varphi \to \psi) \to (j \varphi \to j \psi) \\
\vdash_{\Lo} j \varphi \leftrightarrow j j \varphi \\
\vdash_{\Lo} j(\varphi \to j \psi) \leftrightarrow (j \varphi \to j \psi) \\
\vdash_{\Lo} j (j \varphi \lor j \psi) \leftrightarrow j (\varphi \lor \psi) \\
\vdash_{\Lo} j (\exists x \, j \varphi) \leftrightarrow j (\exists x \varphi) \\
\vdash_{\Lo} j (\forall x \, j \varphi) \leftrightarrow \forall x \, j \varphi
\end{array}
\end{displaymath}
\end{lemm}
\begin{proof}
We only prove the third item, leaving the others (which are easier) to the reader.

Since 
\[ \vdash_{\Lo} ((\varphi \to j\psi) \land \varphi) \to j\psi \]
the third and second axiom for a nucleus imply
\[ \vdash_{\Lo} (j(\varphi \to j\psi) \land j\varphi) \to j\psi \]
which is equivalent to
\[ \vdash_{\Lo} j(\varphi \to j \psi) \to (j \varphi \to j \psi). \]

Conversely, the first axiom for a nucleus implies
\[ \vdash_{\Lo} (\varphi \to j\psi) \to j(\varphi \to j\psi) \]
which in combination with 
\[ \vdash_{\Lo} \varphi \to j\varphi \]
implies 
\[ \vdash_{\Lo}  (j\varphi \to j\psi) \to j(\varphi \to j\psi). \]
\end{proof}

\section{The G\"odel-Gentzen-style $j$-translation}

\begin{defi}{godelgentzenj}
Assume $j$ is a nucleus. We define $\varphi^j$ to be the formula defined by induction on the structure of $\varphi$ as follows:
\begin{eqnarray*}
\varphi^j & := & j \varphi \qquad \mbox{ if } \varphi \mbox{ is an atomic formula } \\
(\varphi \land \psi)^j & := & \varphi^j \land \psi^j \\
(\varphi \lor \psi)^j & := & j (\varphi^j \lor \psi^j) \\
(\varphi \to \psi)^j & := & \varphi^j \to \psi^j \\
(\exists x \, \varphi)^j & := & j ( \exists x \, \varphi^j) \\
(\forall x \, \varphi)^j & := & \forall x \, \varphi^j
\end{eqnarray*}
\end{defi}

Then we have:
\begin{prop}{translation}
\begin{enumerate}
\item[(a)] For any formula $\varphi$ we have $\vdash_{\Lo} j (\varphi^j) \leftrightarrow \varphi^j$.
\item[(b)] $\varphi_1, \ldots, \varphi_n \vdash_{\Lo} \psi$ implies $\varphi_1^j, \ldots, \varphi_n^j \vdash_{\Lo} \psi^j$.
\end{enumerate}
\end{prop}

\begin{prop}{specificnuclei}
\begin{enumerate}
\item[(a)] If $j \varphi = (\varphi \to A) \to A$, then $\varphi_1, \ldots, \varphi_n \vdash_{ \CQC} \psi$ implies \[ \varphi_1^j, \ldots, \varphi_n^j \vdash_{\IQC} \psi^j. \]
\item[(b)] If $j \varphi = (\varphi \to \bot) \to \bot$, then $\varphi_1, \ldots, \varphi_n \vdash_{ \CQC} \psi$ implies $\varphi_1^j, \ldots, \varphi_n^j \vdash_{\MQC} \psi^j$.
\item[(c)] If $j \varphi = \varphi \lor \bot$, then $\varphi_1, \ldots, \varphi_n \vdash_{ \IQC} \psi$ implies $\varphi_1^j, \ldots, \varphi_n^j \vdash_{\MQC} \psi^j$.
\end{enumerate}
\end{prop}

\section{A Kuroda-style $j$-translation}

Our aim in this section is to give an alternative presentation of $\varphi^j$, which brings it closer to Kuroda's negative translation. We do this as follows:

\begin{defi}{kurodastylej}
Let $j$ be a nucleus. First we define $J\varphi$ by induction on $\varphi$, as follows:
\begin{eqnarray*}
J\varphi & ;= & \varphi \mbox{ if } \varphi \mbox{ is an atomic formula } \\
J(\varphi \land \psi) & := & J\varphi \land J\psi \\
J(\varphi \lor \psi) & := & J\varphi \lor J\psi \\
J(\varphi \to \psi) & := & J\varphi \to j(J\psi) \\
J(\exists x \, \varphi) & := & \exists x \, J(\varphi) \\
J(\forall x \, \varphi) & := & \forall x \, j (J\varphi)
\end{eqnarray*}
Finally, we set $\varphi_j := j(J\varphi)$.
\end{defi}

\begin{prop}{equivalence}
\begin{enumerate}
\item[(a)] We have $\vdash_{\Lo} \varphi^j \leftrightarrow \varphi_j$.
\item[(b)] $\varphi_1, \ldots, \varphi_n \vdash_{\Lo} \psi$ implies $J(\varphi_1), \ldots, J(\varphi_n) \vdash_{\Lo} \psi_j$.
\end{enumerate}
\end{prop}
\begin{proof}
Part (a) can be proved by a straightforward induction on the structure of $\varphi$ using \reflemm{basicpropofnuclei}. Part (b) can be proved directly by an induction on the derivation of $\varphi_1, \ldots, \varphi_n \vdash_{\Lo} \psi$, but it also follows from \refprop{translation}.
\end{proof}

\begin{prop}{specificnucleiagain}
\begin{enumerate}
\item[(a)] If $j \varphi = (\varphi \to A) \to A$, then $\varphi_1, \ldots, \varphi_n \vdash_{ \CQC} \psi$ implies \[ J(\varphi_1), \ldots, J(\varphi_n)_j \vdash_{\IQC} \psi_j. \]
\item[(b)] If $j \varphi = (\varphi \to \bot) \to \bot$, then $\varphi_1, \ldots, \varphi_n \vdash_{ \CQC} \psi$ implies $J(\varphi_1), \ldots, J(\varphi_n) \vdash_{\MQC} \psi_j$.
\item[(c)] If $j \varphi = \varphi \lor \bot$, then $\varphi_1, \ldots, \varphi_n \vdash_{ \IQC} \psi$ implies $J(\varphi_1), \ldots, J(\varphi_n) \vdash_{\MQC} \psi_j$.
\end{enumerate}
\end{prop}

\begin{rema}{aboutkuroda} Some nuclei commute with implication in the sense that
\[ \vdash_{\Lo} j(\varphi \to \psi) \leftrightarrow (j \varphi \to j \psi); \]
for instance, the double-negation nucleus commutes with implication in intuitionistic logic, but not in minimal logic. In case the nucleus commutes with implication, the clause for the implication can be modified to
\[ J(\varphi \to \psi) := J\varphi \to J\psi \]
and the previous propositions still hold. Indeed, this is what happens in Kuroda's original negative translation.
\end{rema}

The following corollary is well-known and also has an easy semantic proof using Kripke models (and, indeed, Barr's Theorem gives a much stronger result, see, for example, \cite{maclanemoerdijk94}).

\begin{coro}{classlogicconsforcoherent}
If $\varphi$ and $\psi$ are coherent formulas (that is, formulas not containing universal quantifiers or implications), then $\varphi \vdash_{\CQC} \psi$ implies $\varphi \vdash_{\IQC} \psi$.
\end{coro}
\begin{proof}
Suppose $\varphi \vdash_{\CQC} \psi$ and without loss of generality we may assume that $\varphi$ and $\psi$ are sentences (replace free variables by fresh constants, if necessary). Part (a) from the previous proposition gives us
\[ \varphi \vdash_{\IQC} (\psi \to A) \to A \]
for any sentence $A$. Choosing $A = \psi$ we deduce that $\varphi \vdash_{\IQC} \psi$, as desired.
\end{proof}

The following corollary, however, is much less known. In fact, it seems it was first proved by Johansson in unpublished correspondence with Heyting (see \cite{vandermolen16,colacitoetal17}).

\begin{coro}{anewminimaltranslation}
If $\vdash_{\IQC} \varphi$ then $\vdash_{\MQC} J\varphi$ for the nucleus $j$ given by $j \varphi := \varphi \lor \bot$.
\end{coro}
\begin{proof}
Suppose $\vdash_{\IQC} \varphi$. Part (c) from the previous proposition gives us
\[ \vdash_{\MQC} J(\varphi) \lor \bot \]
for the nucleus $j$ given by $j \varphi := \varphi \lor \bot$. But since minimal logic has the disjunction property and does not prove $\bot$, it follows that $\vdash_{\MQC} J\varphi$, as desired.
\end{proof}

\section{The Kuroda-style $j$-translation in the literature}

The observations in the previous section explain several phenomena which one may observe in the literature. For instance, it explains Beeson's account of forcing in a constructive metatheory in his book \cite{beeson85}.

If $P$ is some partial order and we regard $P$ as an intuitionistic Kripke frame, one can define a nucleus internally to a Kripke model on $P$, as follows:
\begin{eqnarray*} p \Vdash j \varphi & :\Leftrightarrow & (\forall q \leq p) \, (\exists r \leq q) \, r \Vdash \varphi.
\end{eqnarray*}
Now define $p \Vdash_s \varphi$ to mean $p \Vdash J\varphi$ for this internal nucleus $j$. This results in the following clauses:
\begin{eqnarray*}
p \Vdash_s (\varphi \land \psi) & \leftrightarrow & p \Vdash_s \varphi \land p \Vdash_s \psi, \\
p \Vdash_s (\varphi \lor \psi) & \leftrightarrow & p \Vdash_s \varphi \lor p \Vdash_s \psi, \\
p \Vdash_s (\varphi \to \psi) & \leftrightarrow & (\forall q \leq p) \, \big( \, q \Vdash_s \varphi \to (\forall r \leq q) \, (\exists s \leq r) \, s \Vdash_s \psi \, \big)\\
p \Vdash_s \exists x \, \varphi & \leftrightarrow & (\exists x) \, p \Vdash_s \varphi \\
p \Vdash_s \forall x \, \varphi & \leftrightarrow & (\forall x) \, (\forall q \leq p) \, (\exists r \leq q) \, r \Vdash_s \varphi
\end{eqnarray*}
In fact, we can simplify the clause for the implication, because it is equivalent to:
\begin{eqnarray*}
p \Vdash_s (\varphi \to \psi) & \leftrightarrow & (\forall q \leq p) \, \big( \, q \Vdash_s \varphi \to (\exists r \leq q) \, r \Vdash_s \psi \, \big)
\end{eqnarray*}
So it follows from \refprop{equivalence} that we have $p \Vdash \varphi^j$ iff $(\forall q \leq p) \, (\exists r \leq q) \, r \Vdash_s \varphi$, and this is how Beeson uses forcing in his book.

The notation $p \Vdash_s \varphi$ is supposed to recall Cohen's ``strong forcing'' (see \cite{cohen66}). In fact, his notion of strong forcing can be obtained by changing the clause for the implication to
\begin{eqnarray*}
p \Vdash_s (\varphi \to \psi) & \leftrightarrow & (\forall q \leq p) \, \big( \, q \Vdash_s \varphi \to q\Vdash_s \psi \, \big),
\end{eqnarray*}
which works in a classical metatheory, because then one can prove that $j$ commutes with implication (see \refrema{aboutkuroda}). Indeed, in a classical metatheory $p \Vdash j\varphi$ is equivalent to $p \Vdash \lnot\lnot\varphi$. (The connection between Cohen's strong forcing and Kuroda's negative translation was pointed out by Jeremy Avigad \cite{avigad04}.)

Another definition which can be explained by the Kuroda-style $j$-translation is Definition 6.1 in \cite{leevanoosten13}. I do not doubt that some readers are able to come up with other examples as well.

\bibliographystyle{plain} \bibliography{dst}

\end{document}